\documentclass[leqno,CJK]{siamltex704}
\usepackage{amssymb,amsmath,graphicx,amscd,mathrsfs}
\usepackage{color,xcolor,comment,dsfont,pifont}
\usepackage{longtable,graphicx,float,amsfonts,amssymb}
\usepackage{hyperref,multirow,placeins}
\numberwithin{equation}{section}
\def\3bar{{|\hspace{-.02in}|\hspace{-.02in}|}}
\def\E{{\mathcal{E}}}
\def\T{{\mathcal{T}}}
\def\Q{{\mathcal{Q}}}
\def\dQ{{\mathbb{Q}}}

\def\b0{\boldsymbol{0}}

\def\sumT{\sum_{T\in\mathcal{T}_h}}     

\def\bw{{\mathbf{w}}}
\def\bu{{\mathbf{u}}}
\def\bv{{\mathbf{v}}}
\def\bl{{\mathbf{L}}}

\def\bh{{\mathbf{H}}}
\def\bi{{\mathbf{I}}}
\def\bn{{\mathbf{n}}}
\def\be{{\mathbf{e}}}
\def\bf{{\mathbf{f}}}

\def\bP{{\mathbf{P}}}

\def\bx{{\mathbf{x}}}

\newtheorem{theo}{Theorem}[section]
\newtheorem{coro}{Corollary}[section]
\newtheorem{example}{Example}[section]

\newtheorem{algorithm}{Algorithm}

\newcommand{\eps}{\varepsilon}
\newcommand{\di}{\text{div}}

\newcommand{\Real}{\mathbb{R}}

\newcommand{\trb}[1]{|\!|\!|#1|\!|\!|}
\newcommand{\la}{\langle}
\newcommand{\ra}{\rangle_{\partial T}}

\allowdisplaybreaks 

\begin{document}
	
	\title{The two-grid weak Galerkin method and enriched Crouzeix-Raviart element method for linear elastic eigenvalue problems}
	\author{
	Wei Lu and Qilong Zhai
		\thanks{Department of Mathematics, Jilin University, Changchun,
			China (diql15@mails.jlu.edu.cn). }
		}
	\maketitle
	
	\begin{abstract}
	In this paper, we present a two-gird skill to accelerate the weak Galerkin method. By the proper use of parameters, the two-grid weak Galerkin method not only doubles the convergence rate, but also maintains the asymptotic lower bounds property of the weak Galerkin (WG) method. Moreover, we propose an enriched Crouzeix-Raviart (ECR) scheme, which can also provide lower bounds for the linear elastic eigenvalue problems.
	\end{abstract}
	
	\begin{keywords}
		weak Galerkin method, linear elastic eigenvalue problem, locking-free, two-gird method, lower bounds, enriched Crouzeix-Raviart method.
	\end{keywords}
	
	\begin{AMS}
		Primary, 65N30, 65N15, 65N12, 74N20; Secondary, 35B45, 35J50, 35J35
	\end{AMS}

	\section{Introduction}
The eigenvalue problem, especially linear elastic eigenvalue problem, has attracted extensive attention, due to its wide applications in science and engineering \cite{Vogelius1983}. The finite element method (FEM), as an efficient approach to solve PDEs, has been applied to solve many types of eigenvalue problem, such as the Laplacian eigenvalue problem \cite{Liu2013,Grebenkov2013,Luo2012}, Stokes eigenvalue problem \cite{Feng2014,Chen2009} and biharmonic eigenvalue problem \cite{Oden2007c,Gallistl2014a}. Nevertheless, there exists two difficulties when  FEM is applied to solve linear elastic eigenvalue problem. The first one is, when the Poisson ratio $\nu$ is close to $\frac{1}{2}$ ,the the elastic meterials become nearly imcompressible, then the finite element solution may do not converge to the exact solution, which is called “locking” phenomenon \cite{Ainsworth2022,Babuska1992,Ciarlet2022,Bren1992}. The second one is, as discussed in \cite{MR1115240}, standard conforming FEM can only provide upper bounds for eigenvalues due to the minimun-maximum priciple. Therefore, since the eigenvalues are all real numbers, it is important to obtain the lower bounds for eigenvalues in order to get the accurate interval which the exact eigenvalues belong to \cite{Luo2012}.

“Locking” is not a difficult issue to deal with. Recently, many “locking-free” methods have been proposed. For example, mixed methods \cite{Meddahi2013,Lepe2022,Inzunza2023},  nonconforming FEM \cite{Bi2023,Zhang2021,Zhang2024,Zhang2023}, discontinuous Galerkin (DG) methods \cite{MR4542511,Lepe2019}, virtual element method (VEM) \cite{Amigo2023,Zhao2022} and so on. However, many of them fail to provide lower bounds for eigenvalues.

Compared to overcome “Locking”, it is more difficult to obtain the lower bounds for eigenvalues. An efficient way is the construction of nonconforming FEM. Armentano et al. \cite{AD04} obtained asymptotic lower bounds for Laplacian eigenvalue problem by nonconforming Crouzeix-Raviart (CR) element. Later on, Lin et al. \cite{Lin2013a} get asymptotic lower bounds for Laplacian eigenvalue problem by nonconforming ECR element and $EQ_1^{rot}$ element. Furthermore, Xie et al. used CR and ECR element to solve Stokes eigenvalue problem, and managed to obtain explicit lower bounds. In \cite{Zhang2020}, Zhang et al. obtained guaranteed lower bounds for linear elastic eigenvalue problem by the ues of CR element. Frustratingly, it seems to be difficult to construct high order element.

WG method, as a novel FEM proposed in \cite{Wang2013a}, is able to over come the both difficulties mentioned above. It features in the application of weak functions and weak differential operators. Additionally, WG method adopts discontinuous piecewise polynomials on polygonal finite element partitions, making it easy to construct high order element and can be extended to high dimension cases. So far, WG method has been successfully employed to solve various kinds of eigenvalueproblems. For instance, the Laplacian eigenvalue problem \cite{Zhai2019c,Zhai2019b}, Stokes eigenvalue problem \cite{Zhai2023} and Steklov eigenvalue problem \cite{li2024}. But It is worth mentioning that, solving eigenvalue problem costs more time than the corresponding boundary value problem, since it is a semilinear problem actually. Thus, it is necessary to an effective way accelerate the solving speed.

Two-grid method is an efficient skill to solve nonlinear methods. It saves time by solving a nonlinear problem on a coarse grid, and then solve a linear system on a much finer grid. Since propose in \cite{Xu1994}, the two-grid method has been applied to various kinds of problems such as elliptic eigenvalue problems \cite{Yang2011,Zhai2019b}, Stokes eigenvalue problems \cite{Feng2014,Xie2015a} and linear elastic eigenvalue problems \cite{Zhang2024,Zhang2021,Bi2023}.

In this paper, we combine WG method with two-grid method to solve the linear elastic eigenvalue problem. In this way, it can not only save lots of time, but also maintain the locking-free property of WG method \cite{Wang2015a,Huo2024}. What's more, we will show that, by the proper use of mesh sizes, the two-grid WG method can provide asymptotic lower bounds for eigenvalues. Additionally, we propose a mixed method, then solve it by ECR element and based on some results, we obtain asymptotic lower bounds for eigenvalues successfully. In the end, numerical examples will be provided.

The rest of this paper is constructed as follows. In Section 2 we introduce the WG method for the linear elastic eigenvalue problem, and state some basic error estimates. In Section 3, we define some negative norms, and give the corresponding error estimates for the WG method. Section 4 is devoted to the two-gird method. An ECR finite element scheme will be analyzed in Section 5. In the final section, we present some numerical experiments to verify our theoretical analysis.

	\section{A standard discretization of weak Galerkin method}In this section, we state some notations, introduce the standard WG scheme for elastic eigenvalue problem and present some results. Throughout this paper, we always use $C$ to represent a constant independent of Lam\'{e} parameters $\lambda$, and mesh sizes $H$ and $h$, which may have different values accroding to the occurrence. For simplicity, we use $a\lesssim b$ and $a\gtrsim b$ instead of $a\le Cb$ and $a\ge Cb$, respectively.

The standard Sobolev space notations are used in this paper. Let $\Omega\subset\Real^2$ be a bounded domain with $\partial\Omega=\Gamma_D\cup\Gamma_N$, and $\bh^m(\Omega)$ be the Sobolev space. The notations $(\cdot, \cdot)_{m,D}$, $||\cdot||_{m,D}$ and $|\cdot|_{m,D}$ are used as inner-product, norms and seminorms on $\bh^m(D)$, if the region $D$ is an edge of some elements, we use $\langle\cdot,\cdot\rangle _{m,D}$ instead of $(\cdot,\cdot)_{m,D}$. For simplicity, We shall drop the subscript when $m=0$ or $D=\Omega$. Define $\bh_E^1(\Omega)=\{\bv\in\bh^1(\Omega):\bv|_{\Gamma_D}=\b0\}$.

Consider the following linear elastic eigenvalue problem:
\begin{equation}\label{eig}
	\left\{
	\begin{array}{rcl}
		-\nabla\cdot \sigma(\bu) &=& \gamma \bu,\quad \text{in }\Omega,\\
		\bu &=& \b0,\quad~~ \text{on }\Gamma_D,\\
		\sigma(\bu)\bn &=& \b0,\quad~~ \text{on }\Gamma_N,\\
		\int_\Omega \bu^2d\Omega &=& 1,
	\end{array}
	\right.
\end{equation}
where $|\Gamma_D|>0$, $\bn$ is the unit outward normal vector of $\Gamma_N$. The stress tensor $\sigma(\bu)$ is given by
\begin{eqnarray*}
	\sigma(\bu)=2\mu\eps(\bu)+\lambda(\nabla\cdot\bu)\bi,
\end{eqnarray*}
where $\bi\in\Real^{2\times 2}$ is the identity matrix. The strain tensor $\eps(\bu)$ is defined as
\begin{eqnarray*}
	 \eps(\bu)=\frac{1}{2}(\nabla\bu+(\nabla\bu)^T).
\end{eqnarray*}
The Lam\'{e} parameters $\mu$ and $\lambda$ are given by
\begin{eqnarray*}
	\lambda=\frac{E\nu}{(1+\nu)(1-2\nu)}\ \ \ \ \ {\rm and}\ \ \ \ \ \mu=\frac{E}{2(1+\nu)},
\end{eqnarray*}
where $E$ denotes the Young's modulus and $\nu\in(0,0.5)$ is the Poisson ratio.

Let $\T_h$ be a partition of the domain $\Omega$, and the elements
in $\T_h$ are polygons satisfying the regular assumptions specified
in \cite{Wang2014a}. Let $\E_h$ be the edges in $\T_h$, and
$\E_h^0$ denotes by the interior edges $\E_h\backslash\partial\Omega$. For each edge $e\in\E_h^0$, let $\bn_e$ be the unit normal of $e$ pointing from $T^{+}$ to $T^{-}$, the jump of a function $\bv$ through $e$, denoted by $[\![\bv]\!]$, is given by $[\![\bv]\!]|_e=(\bv|_{T^{+}})|_e-(\bv|_{T^{-}})|_e$. For each element $T\in\T_h$, $h_T$ represents the diameter of $T$, and $h=\max\limits_{T\in\T_h} h_T$ denotes the mesh size.

Now we introduce a WG scheme for the eigenvalue problem \eqref{eig}. For a given integer $k\ge1$, define the WG finite element space
\begin{eqnarray*}
	V_h=\big\{\bv=\{\bv_0,\bv_b\}:\bv_0|_T\in \bP_k(T), \bv_b|_e\in \bP_k(e), \forall T\in\T_h, e\in\E_h,\text{ and } \bv_b=\b0 \text{ on }\Gamma_D\big\},
\end{eqnarray*}

Define the sum space $V=V_h+\bh_E^1(\Omega)$. For each $\bv\in V$, we define its weak gradient
$\nabla_w\bv$ and weak strain tensor $\eps_w(\bv)$ as follows.
\begin{definition}
	$\nabla_w\bv|_T$ is the unique polynomial in $[P_{k-1}(T)]^{2\times 2}$ satisfying
	\begin{eqnarray}\label{defw1}
		(\nabla_w\bv,q)_T=-(\bv_0,\nabla\cdot q)_T+\langle \bv_b,q\bn
		\rangle_{\partial T},\quad\forall q\in [P_{k-1}(T)]^{2\times 2},
	\end{eqnarray}
	where $\bn$ denotes the outward unit normal vector and define
	\begin{eqnarray}\label{defw2}
		\eps_w(\bv)=\frac{1}{2}(\nabla_w\bv+(\nabla_w\bv)^T).
	\end{eqnarray}
\end{definition}

For each $\bv\in V$, we define its weak divergence
$\nabla_w\cdot\bv$ as follows.
\begin{definition}
	$\nabla_w\cdot v|_T$ is the unique polynomial in $P_{k-1}(T)$ satisfying
	\begin{eqnarray}\label{defw3}
		(\nabla_w\cdot\bv,\tau)_T=-(\bv_0,\nabla\tau)_T+\langle \bv_b\cdot\bn,\tau
		\rangle_{\partial T},\quad\forall\tau\in P_{k-1}(T),
	\end{eqnarray}
	where $\bn$ denotes the outward unit normal vector.
\end{definition}

For the aim of analysis, some projection operators are also employed in this paper. For each $T\in\T_h$, let $Q_0$
denotes the $L^2$ projection from $\bl^2(T)$ onto $\bP_k(T)$, 
$\dQ_h$
denotes the $L^2$ projection from $[L^2(T)]^{2\times2}$ onto $[P_{k-1}(T)]^{2\times 2}$,
and $\Q_h$ denotes the $L^2$ projection from $L^2(T)$ onto $P_{k-1}(T)$.
For each $e\in\E_h$, let $Q_b$ denotes the $L^2$ projection from $\bl^2(e)$ onto $\bP_k(e)$ for each $e\in\E_h$.
Combining $Q_0$ and $Q_b$ together, we define $Q_h=\{Q_0,Q_b\}$,
which is a projection onto $V_h$.

Next we define three bilinear forms on $V_h$. For any $\bv_h,\bw_h\in V_h$,
\begin{align*}
	s(\bv_h,\bw_h)=&\sumT h_T^{-1+\delta}\langle \bv_0-\bv_b, \bw_0-\bw_b\rangle_{\partial T},\\
	a_w(\bv_h,\bw_h)=&2\mu(\eps_w(\bv),\eps_w(\bw))+\lambda(\nabla_w\cdot \bv,\nabla_w\cdot\bw)+s(\bv,\bw),\\
	b_w(\bv_h,\bw_h)=&(\bv_0,\bw_0),
\end{align*}
where $0<\delta<1$ is a small constant.

Define the following norms on $V_h$ that
\begin{align*}
	|||\bv_h|||^2=a_w(\bv_h,\bv_h),\quad\forall\bv_h\in V_h.
\end{align*}

For the simplicity of notation, we introduce a semi-norm $\|\cdot\|_b$ by
\begin{align*}
	\|\bv_h\|_b^2=b_w(\bv_h,\bv_h),\quad\forall\bv_h\in V_h.
\end{align*}

With these preparations we can give the following WG algorithm.
\begin{algorithm}\label{31}
	Find $\bu_h\in V_h$ and $\gamma_h\in\Real$ such that $\|\bu_h\|_b=1$ and
	\begin{eqnarray}\label{WGscheme}
		a_w(\bu_h,\bv)=\gamma_h b_w(\bu_h,\bv_h),\quad\forall \bv_h\in V_h.
	\end{eqnarray}
\end{algorithm}

For the analysis in this paper, we introduce the following norm on $V$ that
\begin{align*}
	\|\bv\|_V^2=\sumT\|\eps(\bv_0)\|_T^2 + \lambda\sumT\|\nabla_w\cdot\bv\|_T^2 +\sumT h_T^{-1}\|\bv_0-\bv_b\|_{\partial T}^2.
\end{align*}
and the dual norm of $\|\cdot\|_V$ as follows
\begin{align*}
	\|\bv_h\|_{-V}=\sup_{\bw_h\in V, \bw_h\neq\b0}\frac{b_w(\bv_h,\bw_h)}{\|\bw_h\|_V}.
\end{align*}

For the standard WG scheme, the following convergence theorem holds true, and which also gives a lower bound estimate.
\begin{theo}\label{11}
	Suppose $\gamma_{j,h}$ is the j-th eigenvalue of \eqref{WGscheme} and $\bu_{j,h}$ is the corresponding eigenfunction. There exists an exact eigenfuntion $\bu_j$ corresponding to the j-th exact eigenvalue $\gamma_j$ such that the following error estimates hold
	\begin{align*}
		h^{2k}\|\bu_j\|_{k+1}&\lesssim\gamma_j-\gamma_{j,h}\lesssim h^{2k-2\delta}(\|\bu_j\|_{k+1}+\lambda\|\nabla\cdot\bu_j\|_k),\\
		\|\bu_j-\bu_{j,h}\|_V&\lesssim h^{k-\delta}(\|\bu_j\|_{k+1}+\lambda\|\nabla\cdot\bu_j\|_k),\\
		\|\bu_j-\bu_{j,h}\|_b&\lesssim h^{k+1-\delta}(\|\bu_j\|_{k+1}+\lambda\|\nabla\cdot\bu_j\|_k),
	\end{align*}
	when $\bu_j\in\bh_E^{k+1}(\Omega)$ and $h$ is sufficient small.
\end{theo}
	
	\section{Error estimate in negative norm}In this section, we analyze the $\|\cdot\|_V$ error estimate for the WG scheme \eqref{WGscheme}. First, we need to establish the $\|\cdot\|_V$ error estimate for the corresponding boundary value problem. Consider the following linear elaticity equation
\begin{equation}\label{def}
	\left\{
	\begin{array}{rcl}
		-\nabla\cdot \sigma(\bu) &=& \bf,\quad \text{in }\Omega,\\
		\bu &=& 0,\quad \text{on }\Gamma_D,\\
		\sigma(\bu)\bn &=& 0,\quad \text{on }\Gamma_N,
	\end{array}
	\right.
\end{equation}
where $\bf\in\bl^2(\Omega)$.

The WG method is adopted to solve \eqref{def}. For analysis, we define the follwoing norm
\begin{align*}
 	\trb{\bv_h}_{-1}=\sup_{\bw_h\in V, \bw_h\neq\b0}\frac{b_w(\bv_h,\bw_h)}{\trb\bw_h}.
\end{align*}
 
It is easy to check that $\|\cdot\|_V$ is equivalent to $\|\cdot\|_1$ on the space $\bh_E^1(\Omega)$. The relationship between $\|\cdot\|_V$ and $\trb\cdot$ has been discussed in \cite{Wang2013a}, which is presented as follows.
\begin{lemma}
	For any $\bv\in V_h$, there has
	\begin{align*}
		\trb\bv_h\lesssim\|\bv_h\|_V\lesssim h^{-\frac{\delta}{2}}\trb\bv_h.
	\end{align*}
\end{lemma}

The WG method for the boundary value problem \eqref{def} can be described as follows:
\begin{algorithm}Find $\bu_h\in V_h$ such that
	\begin{align}\label{WGsheme2}
		a_w(\bu_h,\bv_h)=b_w(\bf,\bv_h),\quad\forall\bv_h\in V_h.
	\end{align}
\end{algorithm}

Suppose $\bu$ is the exact solution for \eqref{def} and $\bu_h$ is the corresponding numerical solution of \eqref{WGsheme2}. Denote by $\be_h$ the error that
\begin{align*}
	\be_h=Q_h\bu-\bu_h=\{Q_0\bu-\bu_0,Q_b\bu-\bu_b\}.
\end{align*}

Then $\be_h$ satisfies the following equation.
\begin{lemma}\label{33}
	For the error $\be_h$ defined above, we have
	\begin{eqnarray}\label{errequ}
		a_w(\be_h,\bv_h)=\varphi(\bu,\bv_h)+\xi(\bu,\bv_h)+s(Q_h\bu,\bv_h),\quad\forall \bv_h\in V_h,
	\end{eqnarray}
	where
	\begin{align*}
		\varphi(\bu,\bv_h)&=2\mu\sum\limits_{T\in\mathcal{T}_h}\la \bv_0-\bv_b,(\varepsilon(\bu)-\dQ_h(\varepsilon(\bu)))\bn\ra,
		\\
		\xi(\bu,\bv_h)&=\lambda\sum\limits_{T\in\mathcal{T}_h}\la (\bv_0-\bv_b)\cdot\bn,\di\bu-\Q_h(\di\bu)\ra.
	\end{align*}
	Moreover, we have
	\begin{align*}
		a_w(Q_h\bu,\bv_h)=\varphi(\bu,\bv_h)+\xi(\bu,\bv_h)+s(Q_h\bu,\bv_h)+b_w(\bf,\bv_h),\quad\forall \bv_h\in V_h.
	\end{align*}
\end{lemma}
\begin{theo}\label{35}
	Assume the exact solution $\bu$ of \eqref{def} satisfies $\bu\in\bh^{k+1}(\Omega)$ and $\bu_h$ is the numerical solution of the WG scheme \eqref{WGsheme2}. Then the following error estimates holds true,
\begin{align*}
	\trb{Q_h\bu-\bu_h}\lesssim h^{k-\frac{\delta}{2}}(\|\bu\|_{k+1}+\lambda\|\nabla\cdot\bu\|_k).
\end{align*}
\end{theo}

Now, we come to estimate the error $\be_h$ in the norm $\trb{\cdot}_{-1}$. We suppose the partition $\T_h$ is a triangulation, instead of an arbitrary polytopal mesh. The following lemma is crucial in our analysis.
\begin{theo}\label{crucial}
	For each $\bv_h\in V_h$, we have
	\begin{align*}
		\sumT\|\nabla\bv_0\|_T^2 \lesssim \sumT\|\eps(\bv_0)\|_T^2+ \sumT h_T^{-1}\|\bv_0-\bv_b\|_{\partial T}^2 .
	\end{align*}
	Furthermore, we have
	\begin{align*}
		\sumT\|\nabla\bv_0\|_T^2 + \sumT h_T^{-1}\|\bv_0-\bv_b\|_{\partial T}^2\lesssim \|\bv_h\|_V^2.
	\end{align*}
\end{theo}
\begin{proof}
	By the discrete Korn's inequality in \cite{Korn2004}, we have
	\begin{align*}
		\sumT||\nabla\bv_0||_T^2\lesssim& \sumT||\eps(\bv_0)||_T^2+\sum\limits_{e\in\E_h^0}h_e^{-1}||[\![\bv_0]\!]||_e^2\\
		\le&\sumT||\eps(\bv_0)||_T^2+\sumT h_T^{-1}||\bv_0-\bv_b||_{\partial T}^2\\
		\le&\|\bv_h\|_V^2,
	\end{align*}
	which completes the proof.
\end{proof}

The following results is based on \cite{Zhai2019b}.
\begin{lemma}\label{36}
	For any $\bv_h\in V_h$, there exists $\bv\in\bh_E^1(\Omega)$ such that
	\begin{align}\label{6}
		\|\bv\|_1\lesssim\|\bv_h\|_V\quad and\quad \|\bv-\bv_h\|_b\lesssim h\|\bv_h\|_V.
	\end{align}
\end{lemma}
\begin{proof}
	By Lemma 3.5 in \cite{Zhai2019b}, there exists $\bv\in\bh_E^1(\Omega)$ such that
	\begin{align}
		\|\bv\|_1&\lesssim\left(\sumT\|\nabla\bv_0\|_T^2+\sumT\|\bv_0-\bv_b\|_{\partial T}^2\right)^{\frac12},\label{4}\\
		 \|\bv-\bv_h\|_b&\lesssim h\left(\sumT\|\nabla\bv_0\|_T^2+\sumT\|\bv_0-\bv_b\|_{\partial T}^2\right)^{\frac12}.\label{5}
	\end{align}
	Then, combining \eqref{4}-\eqref{5} and Theorem \ref{crucial} leads to \eqref{6}. The proof is complete.
\end{proof}

In order to deduce the error estimate in $\|\cdot\|_{-V}$, we define the following dual problem
\begin{equation}\label{dual}
	\left\{
	\begin{array}{rcl}
		-\nabla\cdot \sigma(\bw) &=& \bv,\quad \text{in }\Omega,\\
		\bw &=& \b0,\quad \text{on }\Gamma_D,\\
		\sigma(\bw)\bn &=& \b0,\quad \text{on }\Gamma_N,
	\end{array}
	\right.
\end{equation}
where $\bv\in\bh_E^1(\Omega)$.
\begin{theo}\label{1}
	Assume $\bu\in\bh^{k+1}(\Omega)$ is the exact solution of \eqref{def} and $\bu_h$ is the numerical solution of the WG scheme \eqref{WGsheme2}. If the solution of the dual problem \eqref{dual} has $H^3(\Omega)$-regularity and $k\ge2$, the following estimate holds true
	\begin{align*}
		\trb{Q_h\bu-\bu_h}_{-1}\lesssim h^{k+2-3\delta/2}(\|\bu\|_{k+1}+\lambda\|\nabla\cdot\bu\|_k).
	\end{align*}
\end{theo}
\begin{proof}
	Denote $e_h=Q_h\bu-\bu_h$. We choose $\bv_h\in V_h$ and $\bv\in\bh_E^1(\Omega)$ such that $\trb\bv_h=1$, $\trb{e_h}_{-1}=b_w(e_h,\bv_h)$, and \eqref{6} holds. It follows from Lemma \ref{33} that
	\begin{align}\label{34}
		a_w(Q_h\bw,\bw_h)=\varphi(\bw,\bw_h)+\xi(\bw,\bw_h)+s(Q_h\bw,\bw_h)+b_w(\bv,\bw_h),\quad\forall \bw_h\in V_h.
	\end{align} 
	
	Taking $\bw_h=Q_h\bw$ in \eqref{errequ} and $\bw_h=e_h$ in \eqref{34}, and subtracting \eqref{34} from \eqref{errequ}, we have
	\begin{align*}
		(e_0,\bv)=\varphi(\bu,Q_h\bw)+\xi(\bu,Q_h\bw)+s(Q_h\bu,Q_h\bw)-\varphi(\bw,e_h)-\xi(\bw,e_h)-s(Q_h\bw,e_h).
	\end{align*}
	
	Since $\bw\in \bh^3(\Omega)$ and $\bu\in\bh^{k+1}(\Omega)$, the following estimates hold
	\begin{align*}
	\varphi(\bu,Q_h\bw)&=2\mu\sum\limits_{T\in\mathcal{T}_h}\la Q_0\bw-Q_b\bw,(\varepsilon(\bu)-\dQ_h(\varepsilon(\bu)))\bn\ra\\
	&\lesssim h^{k+2}\|\bw\|_3\|\bu\|_{k+1},\\
	\xi(\bu,Q_h\bw)&=\lambda\sum\limits_{T\in\mathcal{T}_h}\la (Q_0\bw-Q_b\bw)\cdot\bn,\di\bu-\Q_h(\di\bu)\ra\\
	&\lesssim h^{k+2}\|\bw\|_3(\lambda\|\di\bu\|_k),\\
	s(Q_h\bu,Q_h\bw)&=\sum\limits_{T\in\mathcal{T}_h}h_T^{-1+\delta}\la Q_0\bu-Q_b\bu,Q_0\bw-Q_b\bw\ra\\
	&\lesssim h^{k+2}\|\bw\|_3\|\bu\|_{k+1}.
\end{align*}
Similarly, by Theorem \ref{35}, we have
\begin{align*}
	\varphi(\bw,e_h)+\xi(\bw,e_h)+s(Q_h\bw,e_h)\lesssim h^{k+2-\delta}\|\bv\|_1(\|\bu\|_{k+1}+\lambda\|\di\bu\|_k).
\end{align*}

Thus, from Lemma \ref{36}, we obtain
\begin{align*}
	\trb{e_h}_{-1}&=b_w(e_h,\bv_h)\le (e_0,\bv)+\|e_0\|\|\bv-\bv_h\|\\
	&\lesssim h^{k+2-\delta}\|\bv_h\|_V(\|\bu\|_{k+1}+\lambda\|\di\bu\|_k)\\
	&\lesssim h^{k+2-3\delta/2}(\|\bu\|_{k+1}+\lambda\|\di\bu\|_k),
\end{align*}
which completes the proof.
\end{proof}
\begin{coro}\label{37}
	Under the conditions of Theorem \ref{1}, the following estimate holds true
	\begin{align*}
		\|Q_h\bu-\bu_h\|_{-V}\lesssim h^{k+2-3\delta/2}(\|\bu\|_{k+1}+\lambda\|\nabla\cdot\bu\|_k).
	\end{align*}
\end{coro}
\begin{lemma}\label{2}
	When $\bu\in\bh^{k+1}(\Omega)$, the following estimate holds true
	\begin{align*}
		\|Q_h\bu-\bu\|_{-V}\lesssim h^{k+2}\|\bu\|_{k+1}.
	\end{align*}
\end{lemma}
\begin{proof}
	See \cite{Zhai2019b}.
\end{proof}

Combining Corollary \ref{37} with Lemma \ref{2}, we have the following error estimate result for the boundary value problem \eqref{def}.
\begin{theo}\label{3}
	Under the conditions of Theorem \ref{1}, the following estimate holds true
	\begin{align*}
		\|\bu-\bu_h\|_{-V}\lesssim h^{k+2-3\delta/2}(\|\bu\|_{k+1}+\lambda\|\nabla\cdot\bu\|_k).
	\end{align*}
\end{theo}

From the Babu\u{s}ka-Osborn's theory, the conclusion of Theorem \ref{3} can be extended to the eigenvalue problem, which means we have the following estimate.
\begin{theo}
	Suppose the solution of the dual problem \eqref{dual} has $H^3(\Omega)$-regularity and $k\ge2$, $(\gamma_{j,h},\bu_{j,h})$ is the $j$-th eigenpair of \eqref{WGscheme}. Then there exists an exact eigenfunction $\bu_j$ corresponding to the $j-$th exact eigenvalue of \eqref{eig} such that the following error estimate holds
	\begin{align*}
		\|\bu_j-\bu_{j,h}\|_{-V}\lesssim h^{k+2-3\delta/2}(\|\bu_j\|_{k+1}+\lambda\|\nabla\cdot\bu_j\|_k),
	\end{align*}
	when $\bu_j\in\bh^{k+1}(\Omega)$.
\end{theo}
	
	\section{A two-grid scheme}In this section, we propose a two-grid WG scheme for the eigenvalue problem, and give the corresponding analysis for the convergence and efficiency of this scheme. Here, we drop the subscript $j$ to denote a certain eigenvalue of problem \eqref{eig}.
\begin{algorithm}Step1:Generate a coarse grid $\T_H$ on the domain $\Omega$ and solve the following eigenvalue problem on the coarse grid $\T_H$:
	
	Find $\gamma_H\in\Real$ and $\bu_H\in V_H$ such that
	\begin{align*}
		a_s(\bu_H,\bv_H)=\gamma_H b_w(\bu_H,\bv_H),\quad\forall\bv_H\in V_H.
	\end{align*}
	Step2: Refine the coarse grid $T_H$ to obtain a finer grid $\T_h$ and solve one single linear problem on the fine grid $\T_h$:
	
	Find $\widetilde{\bu}_h\in V_h$ such that
	\begin{align*}
		a_s(\widetilde{\bu}_h,\bv_h)=\gamma_H b_w(\bu_h,\bv_h),\quad\forall\bv_h\in V_h.
	\end{align*}
	Step3: Calculate the Rayleigh quotient for $\bu_h$
	\begin{align*}
		\widetilde{\gamma}_h=\frac{a_s(\widetilde{\bu}_h,\widetilde{\bu}_h)}{b_w(\widetilde{\bu}_h,\widetilde{\bu}_h)}.
	\end{align*}
	Finally, we obtain the eigenpair approximation $(\widetilde{\gamma}_h,\widetilde{\bu}_h)$.
\end{algorithm}

First, we need the following discrete Poincar\'{e}'s inequality for the WG method, which has been proved in \cite{Wang2014}.
\begin{lemma}
	The discrete Poincar\'{e}'s inequality holds true on $V_h$, i.e.
	\begin{eqnarray*}
		\|\bv_h\|_b\lesssim |||\bv_h|||,\quad\forall\bv_h\in V_h.
	\end{eqnarray*}
\end{lemma}

From Theorem \ref{11}, suppose the eigenfunction $\bu$ is smooth enough and we have the following estimate immediately
\begin{eqnarray*}
	h^{2k}\lesssim \gamma-\gamma_h\lesssim h^{2k-2\delta}.
\end{eqnarray*}

For simplicity, here and hereafter, we assume the concerned eigenvalues are simple. In order to estimate $|\gamma-\widetilde{\gamma}_h|$, we just need to estimate $|\gamma_h-\widetilde{\gamma}_h|$.

\begin{lemma}\label{12}
	Suppose $(\widetilde{\gamma}_h,\widetilde{\bu}_h)$ is calculated by Algorithm 3 and $(\gamma_h,\bu_h)$ satisfies \eqref{WGscheme}. Then the following estimate holds
	\begin{eqnarray*}
		|\gamma_h-\widetilde{\gamma}_h|\lesssim |||\widetilde{\bu}_h-\bu_h|||^2.
	\end{eqnarray*}
\end{lemma}

\begin{lemma}\label{14}
	Under the conditions of Lemma \ref{12}, the following estimate holds true
	\begin{eqnarray}\label{13}
		|||\widetilde{\bu}_h-\bu_h|||\lesssim H^{2k-2\delta}+H^{k+2-2\delta}, \quad when\quad h<H.
	\end{eqnarray}
\end{lemma}
\begin{proof}
	By Lemma 4.3 in \cite{Zhai2019b}, we have
	\begin{align*}
		a_s(\widetilde{\bu}_h-\bu_h,\bv_h)&=\gamma_Hb_w(\bu_H-\bu,\bv_h)+\gamma_Hb_w(\bu-\bu_h.\bv_h)\\
		&\quad+(\gamma_H-\gamma)b_w(\bu_h.\bv_h)+(\gamma-\gamma_h)b_w(\bu_h,\bv_h).
	\end{align*}
	If $k=1$ or the solution of the dual problem \eqref{dual} has the $H^2(\Omega)$-regularity, we have
	\begin{align*}
		a_s(\widetilde{\bu}_h-\bu_h,\bv_h)&\lesssim(\|\bu-\bu_H\|_b+\|\bu-\bu_h\|_b)\|\bv_h\|_b+(|\gamma_H-\gamma|+|\gamma_h-\gamma|)\|\bv_h\|_b\\
		&\lesssim (H^{k+1-\delta}+h^{k+1-\delta})\|\bv_h\|_b+(H^{2k-2\delta}+h^{2k-2\delta})\|\bv_h\|_b\\
		&\lesssim (H^{k+1-\delta}+H^{2k-2\delta})|||\bv_h|||\\
		&\lesssim H^{2k-2\delta}|||\bv_h|||.
	\end{align*}
	
	If $k>1$ and the solution of the dual problem \eqref{dual} has the $H^3(\Omega)$-regularity, we have
	\begin{align*}
		a_s(\widetilde{\bu}_h-\bu_h,\bv_h)&\lesssim(\|\bu-\bu_H\|_{-V}+\|\bu-\bu_h\|_{-V})\|\bv_h\|_V+(|\gamma_H-\gamma|+|\gamma_h-\gamma|)\|\bv_h\|_b\\
		&\lesssim (H^{k+2-3\delta/2}+h^{k+2-3\delta/2})\|\bv_h\|_V+(H^{2k-2\delta}+h^{2k-2\delta})\|\bv_h\|_b\\
		&\lesssim (H^{k+2-2\delta}+H^{2k-2\delta})|||\bv_h|||\\
		&\lesssim H^{k+2-2\delta}|||\bv_h|||.
	\end{align*}
	
	By substituting $\bv_h=\widetilde{\bu}_h-\bu_h$ into the estimates above, we can obtain the desired result \eqref{13} and the proof is completed.
\end{proof}

From Lemma \ref{12} and \ref{14}, the convergence of $\gamma_h-\widetilde{\gamma}_h$ follows immediately.

\begin{lemma}\label{15}
	Suppose $(\widetilde{\gamma}_h,\widetilde{\bu}_h)$ is calculated by Algorithm 3 and $(\gamma_h,\bu_h)$ satisfies \eqref{WGscheme}. Then the following estimate holds
	\begin{eqnarray*}
		|\gamma_h-\widetilde{\gamma}_h|\lesssim H^{4k-4\delta}+H^{2k+4-4\delta},\quad when\quad h<H.
	\end{eqnarray*}
\end{lemma}

With Lemma \ref{14} and \ref{15}, we arrive at the following convergence theorem.

\begin{theo}\label{16}
	Suppose $(\widetilde{\gamma}_h,\widetilde{\bu}_h)$ is calculated by Algorithm 3, $h<H$ and the exact eigenfunctions of \eqref{eig} have $H^{k+1}(\Omega)$-regularity. Then there exists an exact eigenpair $(\gamma,\bu)$ such that the following estimates hold true
	\begin{align*}
		|||Q_h\bu-\widetilde{\bu}_h|||&\lesssim H^{\bar{k}}+h^{k-\delta/2},\\
		|\gamma-\widetilde{\gamma}_h|&\lesssim H^{2\bar{k}}+h^{2k-2\delta},
	\end{align*}
	where $\bar{k}=\min\{2k-2\delta,k+2-2\delta\}$.
\end{theo}

From Theorem \ref{16} and Lemma \ref{15}, we can get the following lower bound estimate.
\begin{theo}
	Suppose the conditions of Theorem \ref{16} hold. Let $\bar{k}=\min\{2k-2\delta,k+2-2\delta\}$ and $\delta_0$ be a positive number. If $H^{2\bar{k}}\lesssim h^{2k+\delta_0}$, then when $H$ and $h$ are sufficiently small, we have
	\begin{align*}
		\widetilde{\gamma}_h\le\gamma.
	\end{align*}
\end{theo}
\begin{proof}
	From Theorem \ref{11} we have
	\begin{align*}
		\gamma-\gamma_h\gtrsim h^{2k}.
	\end{align*}
	
	According to Lemma \ref{15}, the following estimate holds
	\begin{align*}
	|\gamma_h-\widetilde{\gamma}_h|\lesssim H^{2\bar{k}}\lesssim h^{2k+\delta_0}.
	\end{align*}
	
	Then, when $h$ is sufficiently small, we obtain
	\begin{align*}
		\gamma-\widetilde{\gamma}_h&=\gamma-\gamma_h+\gamma_h-\widetilde{\gamma}_h\ge\gamma-\gamma_h-|\gamma_h-\widetilde{\gamma}_h|\\
		&\gtrsim h^{2k}-h^{2k+\delta_0}\ge 0,
	\end{align*}
	which completes the proof.
\end{proof}
	
	\section{An Enriched Crouzeix-Raviart element scheme}In this section, we consider the following linear elastic eigenvalue problem:
\begin{equation}\label{eig2}
	\left\{
	\begin{array}{rcl}
		-\mu\triangle\bu-(\lambda+\mu)\nabla(\di\bu) &=& \gamma \bu,\quad \text{in }\Omega,\\
		\bu &=& \b0, \quad~~ \text{on }\partial\Omega,\\
		\int_\Omega \bu^2d\Omega &=& 1.
	\end{array}
	\right.
\end{equation}

Let $p=(\lambda+\mu)\di\bu$, we can obtain the following equivalent problem:
\begin{equation}\label{eig3}
	\left\{
	\begin{array}{rcl}
		-\mu\Delta\bu-\nabla p &=& \gamma \bu,\quad \text{in }\Omega,\\
		\nabla\cdot\bu-\frac{1}{\lambda+\mu}p&=&0,~~\quad \text{in }\Omega,\\
		\bu &=& \b0, \quad~~ \text{on }\partial\Omega,\\
		\int_\Omega \bu^2d\Omega &=& 1.
	\end{array}
	\right.
\end{equation}

Denote $L_0^2(\Omega)=\{q\in L^2(\Omega): \int_\Omega qd\Omega=0\}$. Then the weak formulation of \eqref{eig3} can be written to find $\gamma\in\Real$,  $\bu\in\bh_0^1(\Omega)$ and $p\in L_0^2(\Omega)$ such that $(\bu,\bu)=1$ and
\begin{equation}\label{eig-form}
	\left\{
	\begin{array}{rcl}
		a(\bu,\bv)+b(\bv,p)&=&\gamma(\bu,\bv),\quad\forall\bv\in\bh_0^1(\Omega),\\
		b(\bu,q)-d(p,q)&=&0,\quad\quad\quad~~\forall q\in L^2(\Omega),
	\end{array}
	\right.
\end{equation}
where
\begin{align*}
	a(\bu,\bv) &=\mu(\nabla\bu,\nabla\bv),\\
	b(\bv,q) &=(\nabla\cdot\bv,q),\\
	d(p,q)&=\frac{1}{\lambda+\mu}(p,q).
\end{align*}

Consider a regular triangular mesh $\T_h$ that partition $\Omega$ into triangles, $\E_h$ denotes the set of all edges of $\T_h$. Then we define the following two finite element spaces on $\T_h$:
\begin{align*}
	U_h=&\left\{v\in L^2(\Omega):v|_T\in span\{1,x,y,x^2+y^2\},\int_e v|_{T_1}ds=\int_e v|_{T_2}ds\right.\\
	&\quad\left.\text{if } e=T_1\cap T_2,~ and \int_e v|_Tds=0~\text{if }e=T\cap\partial\Omega\right\},
\end{align*}
\begin{align*}
	W_h=&\left\{v\in L^2(\Omega):v|_T\in span\{1\},~\forall T\in\T_h\right\},
\end{align*}
where $U_h$ is the ECR finite element space.

Denote $V_h=U_h^2$, and the corresponding interpolation operator $I_h:\bh_0^1(\Omega)\rightarrow V_h$ is defined by
\begin{eqnarray}
	\int_{e}(\bv-I_h\bv)ds=\b0,\quad\forall e\in\E_h,\label{int1}\\
	\int_{T}(\bv-I_h\bv)d\bx=\b0,\quad\forall T\in\T_h.\label{int2}
\end{eqnarray}

The following interpolation estimate can be found in \cite{Hu2014b,Lin2013a}.
\begin{lemma}
	For any $\bu\in\bh^{1+s}(\Omega)(0<s\le1)$, there has
	\begin{eqnarray}\label{24}
		||\bu-I_h\bu||_0+h||\bu-I_h\bu||_h\lesssim h^{1+s}||\bu||_{1+s}.
	\end{eqnarray}
\end{lemma}

Now we are ready to introduce the ECR finite element scheme.
\begin{algorithm}\label{32}
	Find $\gamma_h\in\Real$, $\bu_h\in V_h$ and $q_h\in W_h$ such that $(\bu_h,\bu_h)=1$ and
	\begin{equation}
		\left\{
		\begin{array}{rcl}
			a_h(\bu_h,\bv_h)+b(\bu_h,p_h)&=&\gamma_h(\bu_h,\bv_h),\quad\forall\bv\in V_h,\\
			b_h(\bv_h,q_h)-d_h(p_h,q_h)&=&0,\quad\quad\quad\quad\quad\forall q_h\in W_h,
		\end{array}
		\right.
	\end{equation}
	where
	\begin{align*}
		a_h(\bu,\bv_h) &=\mu\sumT(\nabla\bu_h,\nabla\bv_h)_T,\\
		b_h(\bv,q_h) &=\sumT(\nabla\cdot\bv_h,q_h)_T,\\
		d_h(p_h,q_h)&=\frac{1}{\lambda+\mu}\sumT(p_h,q_h)_T.
	\end{align*}
\end{algorithm}

For each $\bv\in V_h+\bh_0^1(\Omega)$, its norm $||\cdot||_h$ is defined by
\begin{eqnarray}
	||\bv||_h=\sqrt{a_h(\bv_h,\bv_h)},
\end{eqnarray} 
\begin{lemma}
	The following inf-sup condition holds for the space $V_h\times W_h$:
	\begin{align*}
		\sup\limits_{\b0\neq\in V_h}\frac{b_h(\bv_h,q_h)}{\|\bv_h\|_h}\gtrsim \|q_h\|,\quad\forall q_h\in W_h.
	\end{align*}
\end{lemma}
\begin{proof}
	See \cite{Liu2018,Lin2010,Crouzeix1973}.
\end{proof}

The following error estimates are based on the results of \cite{MR1115240,Boffi2000,Lin2012,Osborn1976}.
\begin{lemma}\label{22}
	Suppose the exact eigenpair $(\gamma,\bu,p)$ of \eqref{eig3} satisfies $\bu\in\bh^{1+s}(\Omega)$ and $p\in H^s(\Omega)(0<s\le1)$, then for the eigenpair $(\gamma_h,\bu_h,p_h)$ obtained by \eqref{eig-form}, there exists an exact eigenpair $(\gamma,\bu,p)$ such that
	\begin{align}
		\|\bu-\bu_h\|_h+\|p-p_h\|&\lesssim h^s(\|\bu\|_{1+s}+\|p\|_s),\\
		\|\bu-\bu_h\|&\lesssim h^s\|\bu-\bu_h\|_h\lesssim h^{2s}(\|\bu\|_{1+s}+\|p\|_s).\label{25}
	\end{align}	
\end{lemma}
\begin{theo}\label{23}
	Let $(\omega,\bu,p)$ be the eigenpair of \eqref{eig3}, and $(\omega_h,\bu_h,p_h)$ be eigenpair of \eqref{eig-form}, then we have the following expansion
	\begin{eqnarray}\label{17}
		\quad\quad\quad\gamma-\gamma_h=\|\bu-\bu_h\|_h^2+\|p-p_h\|^2-\gamma_h\|I_h\bu-\bu_h\|^2+\gamma_h(\|I_h\bu\|^2-\|\bu\|^2).
	\end{eqnarray}
\end{theo}
\begin{proof}
	By $b(\bu,\bu)=1$ and $b_h(\bu_h,\bu_h)=1$, we have
	\begin{align*}
		\|\bu\|_h^2+\|p\|^2=\gamma,\quad\|\bu_h\|_h^2+\|p_h\|^2=\gamma_h.
	\end{align*}
	Then for each $\bv_h\in V_h$, we have
	\begin{align}\label{18}
		\|\bu-\bu_h\|_h^2&=\|\bu\|_h^2+\|\bu_h\|_h^2-2a_h(\bu,\bu_h)\\
		&=\gamma+\gamma_h-(\|p\|^2+\|p_h\|^2)-2a_h(\bu-\bv_h,\bu_h)-2a_h(\bv_h,\bu_h).\notag
	\end{align}
	By the first equation of \eqref{eig-form}, we obtain
	\begin{align}\label{19}
		&\quad-2a_h(\bv_h,\bu_h)-2b_h(\bv_h,p_h)=-2\gamma_h(\bv_h,\bu_h)\notag\\
		&=\gamma_h\|\bv_h-\bu_h\|^2-\gamma_h\|\bv_h\|^2-\gamma_h\|\bu_h\|^2\\
		&=\gamma_h\|\bv_h-\bu_h\|^2-\gamma_h(\|\bv_h\|^2-\|\bu_h\|^2)-2\gamma_h\notag,
	\end{align}
	and the second equation of \eqref{eig-form} implies 
	\begin{align}\label{20}
		b_h(\bu,p_h)=(p,p_h).
	\end{align}
	Combining \eqref{18}-\eqref{20}, we derive
	\begin{align}\label{21}
		\|\bu-\bu_h\|_h^2-2b_h(\bv_h-\bu,p_h)&=\gamma-\gamma_h-\|p-p_h\|^2-2a_h(\bu-\bv_h,\bu_h)\\
		&\quad+\gamma_h\|\bv_h-\bu_h\|^2-\gamma_h(\|\bv_h\|^2-\|\bu_h\|^2)\notag
	\end{align}
	Substituting $\bv_h=I_h\bu$ into \eqref{21} we have
	\begin{align*}
		\gamma-\gamma_h&=\|\bu-\bu_h\|_h^2+\|p-p_h\|^2-\gamma_h\|I_h\bu-\bu_h\|^2\\
		&\quad+\gamma_h(\|I_h\bu\|^2-\|\bu\|^2)+2a_h(\bu-I_h\bu,\bu_h)-2b_h(I_h\bu-\bu,p_h).
	\end{align*}
	By \eqref{int1}-\eqref{int2} and Green formulation, it is easy to check
	\begin{align*}
		a_h(\bu-I_h\bu,\bu_h)=0,\quad b_h(I_h\bu-\bu,p_h)=0,
	\end{align*}
	which completes the proof.
\end{proof}

Next, we are ready to proof the lower bound property of \eqref{eig-form}.
\begin{theo}
	Suppose the conditions of Lemma \ref{22} hold, then when h is small enough, we have
	\begin{eqnarray}\label{30}
		\gamma-\gamma_h\ge 0.
	\end{eqnarray}
\end{theo}
\begin{proof}
	It follow from Theorem \ref{23} that
	\begin{align}\label{26}
		\gamma-\gamma_h\ge\|\bu-\bu_h\|_h^2-\gamma_h\|I_h\bu-\bu_h\|^2+\gamma_h(\|I_h\bu\|^2-\|\bu\|^2).
	\end{align}
	For the first term, by Theorem 2.1 in \cite{Lin2011a}, the following estimate holds:
	\begin{align}\label{27}
		\|\bu-\bu_h\|_h^2\gtrsim h^2.
	\end{align}
	For the second term, from \eqref{24} and \eqref{25}, we have
	\begin{align}\label{28}
		\|I_h\bu-\bu_h\|^2\le \|I_h\bu-\bu\|^2+\|\bu-\bu_h\|^2\lesssim h^{2(1+s)}+h^{2s}\|\bu-\bu_h\|_h^2.
	\end{align}
	For the third term, by Theorem 2.4 in \cite{Luo2012}, we obtain
	\begin{align}\label{29}
		\left|\|I_h\bu\|^2-\|\bu\|^2\right.|\lesssim h^{2+s}.
	\end{align}
	Substistuting \eqref{27}-\eqref{29} into \eqref{26}, we arrive at \eqref{30}. The proof is complete.
\end{proof}
	
	\section{Numerical experiments}In this section, we present some numerical examples of Algorithms 3 and 4 to check the efficiencies and lower bound properties of Algorithms 3 and 4 for the eigenvalue problem \eqref{eig}. In the following examples, uniform mesh is applied, $H$ and $h$ denote mesh sizes. We set the Young's modulus $E$=1, $\delta=0.1$, and report the first five discrete eigenfrequencies $\omega_h=\sqrt{\gamma_h}$. Since the exact eigenvalues are unknown, we compute the convergence rate by the following estimate
\begin{eqnarray*}
	Order\approx\lg\left(\frac{\gamma_h-\gamma_{\frac{h}{2}}}{\gamma_{\frac{h}{2}}-\gamma_{\frac{h}{4}}}\right)/\lg2.
\end{eqnarray*}

\begin{example}\label{e1}
	Consider the linear elastic eigenvalue problem \eqref{eig} on unit square domain $\Omega=(0,1)^2$ with $\Gamma_N=\phi$. We set the Poisson ratio $\nu$=0.49, 0.4999, 0.499999, and solve it by Algorithms \ref{31} and \ref{32}. The corresponding results are presented in table \ref{t1}-\ref{t3}.
\end{example}
\begin{table}[H]
	\begin{center}
		\caption{$WG$ method, k=1}\label{t1}
		\begin{tabular}{|c|c|c|c|c|}\hline
			$H$&1/8&1/16&1/32&Order\\ \hline
			$h$&1/16&1/64&1/256&\\ \hline\hline
			\multicolumn{5}{|c|}{$\nu=0.49$}\\ \hline
			$\omega_{1,h}$&4.126189 & 4.183792 & 4.188228 & \textbf{3.70} \\ \hline
			$\omega_{2,h}$&5.335989 & 5.503344 & 5.516541 & \textbf{3.66} \\ \hline
			$\omega_{3,h}$&5.344595 & 5.504092 & 5.516599 & \textbf{3.67} \\ \hline
			$\omega_{4,h}$&6.317241 & 6.525546 & 6.542092 & \textbf{3.65} \\ \hline
			$\omega_{5,h}$&6.733102 & 7.105710 & 7.135246 & \textbf{3.66} \\ \hline
			\multicolumn{5}{|c|}{$\nu=0.4999$}\\ \hline
			$\omega_{1,h}$&4.114974 & 4.172341 & 4.176771 & \textbf{3.69} \\ \hline
			$\omega_{2,h}$&5.358253 & 5.527023 & 5.540415 & \textbf{3.66} \\
		    \hline
			$\omega_{3,h}$&5.367297 & 5.527753 & 5.540472 & \textbf{3.66} \\
		    \hline
			$\omega_{4,h}$&6.311732 & 6.519471 & 6.536049 & \textbf{3.65} \\ \hline
			$\omega_{5,h}$&6.759380 & 7.135193 & 7.165288 & \textbf{3.64} \\ \hline
		    \multicolumn{5}{|c|}{$\nu=0.499999$}\\ \hline
			$\omega_{1,h}$&4.114964 & 4.172331 & 4.176761 & \textbf{3.69} \\ \hline
			$\omega_{2,h}$&5.358269 & 5.527039 & 5.540432 & \textbf{3.66} \\
			\hline
			$\omega_{3,h}$&5.367312 & 5.527769 & 5.540489 & \textbf{3.66} \\
			\hline
			$\omega_{4,h}$&6.311726 & 6.519465 & 6.536044 & \textbf{3.65} \\ \hline
			$\omega_{5,h}$&6.759400 & 7.135215 & 7.165311 & \textbf{3.64} \\ \hline			
		\end{tabular}
	\end{center}
\end{table}

\begin{table}[H]
	\begin{center}
		\caption{$WG$ method, k=2}\label{t2}
		\begin{tabular}{|c|c|c|c|c|}\hline
			$H$&1/8&1/16&1/32&Order\\ \hline
			$h$&1/16&1/64&1/256&\\ \hline\hline
			\multicolumn{5}{|c|}{$\nu=0.49$}\\ \hline
			$\omega_{1,h}$&4.188132 & 4.188575 & 4.188577 & \textbf{7.50} \\ \hline
			$\omega_{2,h}$&5.515782 & 5.517571 & 5.517581 & \textbf{7.49} \\ \hline
			$\omega_{3,h}$&5.515718 & 5.517572 & 5.517581 & \textbf{7.63} \\ \hline
			$\omega_{4,h}$&6.539853 & 6.543344 & 6.543362 & \textbf{7.56} \\ \hline
			$\omega_{5,h}$&7.131011 & 7.137495 & 7.137527 & \textbf{7.64} \\ \hline
			\multicolumn{5}{|c|}{$\nu=0.4999$}\\ \hline
			$\omega_{1,h}$&4.176665 & 4.177117 & 4.177119 & \textbf{7.85} \\ \hline
			$\omega_{2,h}$&5.539595 & 5.541463 & 5.541473 & \textbf{7.54} \\
			\hline
			$\omega_{3,h}$&5.539523 & 5.541464 & 5.541473 & \textbf{7.74} \\
			\hline
			$\omega_{4,h}$&6.533747 & 6.537305 & 6.537324 & \textbf{7.57} \\ \hline
			$\omega_{5,h}$&7.160882 & 7.167588 & 7.167621 & \textbf{7.67} \\ \hline
		\end{tabular}
	\end{center}
\end{table}

In Example \ref{e1}, the eigenfuntions corresponding to first 5 eigenvalues are smooth. As we can see from Table \ref{t1} and \ref{t2}, the convergence rate of eigenvalues with polynomial degree $k$=1, 2 are approximately 2$(k-\delta)$, which concides with Theorem \ref{16}. Furthermore, in Table \ref{t1}, by setting the mesh size $h=(2H)^2$, which means the conditions of Theorem \ref{17} are satisfied, we manage to obtain the lower bounds for eigenvalues. To our surprise, as shown in Table \ref{t2}, although the choince of $k=2$ and $\delta=0.1$ do not satisfy the conditions of Theorem \ref{17}, the two-grid WG method still manages to provide lower bounds for eigenvalues. Additionally, by Table \ref{t3}, We can find the eigenvalue approximations of Algorithm \ref{32} reach the optimal convergence orderare all lower bounds of the exact eigenvalues.

\begin{table}[H]
	\begin{center}
		\caption{$ECR$ method}\label{t3}
		\begin{tabular}{|c|c|c|c|c|c|c|}\hline
			$h$&1/16&1/32&1/64&1/128&1/256&Order\\ \hline\hline
			\multicolumn{7}{|c|}{$\nu=0.49$}\\ \hline
			$\omega_{1,h}$&4.163418 & 4.181869 & 4.186856 & 4.188142 & 4.188468 & \textbf{1.98} \\ \hline
			$\omega_{2,h}$&5.436176 & 5.496398 & 5.512192 & 5.516223 & 5.517241 & \textbf{1.99} \\ \hline
			$\omega_{3,h}$&5.445379 & 5.498658 & 5.512754 & 5.516364 & 5.517276 & \textbf{1.98} \\ \hline
			$\omega_{4,h}$&6.437887 & 6.515161 & 6.536119 & 6.541532 & 6.542903 & \textbf{1.98} \\ \hline
			$\omega_{5,h}$&6.962361 & 7.092108 & 7.125983 & 7.134619 & 7.136798 & \textbf{1.99} \\ \hline
		\multicolumn{7}{|c|}{$\nu=0.4999$}\\ \hline
			$\omega_{1,h}$&4.151835 & 4.170450 & 4.175484 & 4.176782 & 4.177111 & \textbf{1.98} \\ \hline
			$\omega_{2,h}$&5.463024 & 5.520883 & 5.536107 & 5.539997 & 5.540979 & \textbf{1.99} \\ \hline
			$\omega_{3,h}$&5.469977 & 5.522596 & 5.536533 & 5.540104 & 5.541006 & \textbf{1.98} \\ \hline
			$\omega_{4,h}$&6.432080 & 6.509186 & 6.530137 & 6.535550 & 6.536922 & \textbf{1.98} \\ \hline
			$\omega_{5,h}$&6.995574 & 7.122888 & 7.156087 & 7.164550 & 7.166685 & \textbf{1.99} \\ \hline
		\multicolumn{7}{|c|}{$\nu=0.499999$}\\ \hline
			$\omega_{1,h}$&4.151720 & 4.170337 & 4.175372 & 4.176999 & 4.176998 & \textbf{1.98} \\ \hline
			$\omega_{2,h}$&5.463233 & 5.521072 & 5.536290 & 5.541161 & 5.541157 & \textbf{1.99} \\ \hline
			$\omega_{3,h}$&5.470167 & 5.522780 & 5.536715 & 5.541188 & 5.541184 & \textbf{1.98} \\ \hline
			$\omega_{4,h}$&6.432019 & 6.509123 & 6.530074 & 6.536859 & 6.536857 & \textbf{1.98} \\ \hline
			$\omega_{5,h}$&6.995853 & 7.123141 & 7.156333 & 7.164795 & 7.166929 & \textbf{1.99} \\ \hline
		\end{tabular}
	\end{center}
\end{table}

In Example \ref{e2} and \ref{e3}, although some eigenfunctions are sigular, we can see from Table \ref{t4} and \ref{t5}, the two-grid WG method still provides the lower bounds for eigenvalues with polynomial degree $k$=1, 2 and mesh size $h=(2H)^2$. What's more, since there is no limitation for $s$ in Lemma \ref{22}, it is reasonable to get lower bounds for eigenvalues by Algorithm \ref{31} even though the corresponding eigenfunctions have low regularity.

\begin{example}\label{e2}
	Consider the linear elastic eigenvalue problem \eqref{eig} on unit square domain $\Omega=(0,1)^2$ with $\Gamma_D=\{(x,0):0\le x\le1\}$. We set tthe Poisson ratio $\nu$=0.49, 0.4999, 0.499999, and solve it by Algorithm \ref{31}. The corresponding results are shown in table \ref{t4}-\ref{t5}.
\end{example}
\begin{table}[H]
	\begin{center}
		\caption{$WG$ method, k=1}\label{t4}
		\begin{tabular}{|c|c|c|c|c|}\hline
			$H$&1/8&1/16&1/32&Trend\\ \hline
			$h$&1/16&1/64&1/256&\\ \hline\hline
			\multicolumn{5}{|c|}{$\nu=0.49$}\\ \hline
			$\omega_{1,h}$&0.684447 & 0.696457 & 0.698899 & $\nearrow$ \\ \hline
			$\omega_{2,h}$&1.812638 & 1.831746 & 1.836036 & $\nearrow$ \\ \hline
			$\omega_{3,h}$&1.837057 & 1.859167 & 1.860697 & $\nearrow$ \\ \hline
			$\omega_{4,h}$&2.858137 & 2.914030 & 2.924182 & $\nearrow$ \\ \hline
			$\omega_{5,h}$&3.000610 & 3.040205 & 3.043163 & $\nearrow$ \\ \hline
			\multicolumn{5}{|c|}{$\nu=0.4999$}\\ \hline
			$\omega_{1,h}$&0.685879 & 0.698313 & 0.701075 & $\nearrow$ \\ \hline
			$\omega_{2,h}$&1.822929 & 1.842612 & 1.847548 & $\nearrow$ \\ \hline
			$\omega_{3,h}$&1.841151 & 1.863890 & 1.865497 & $\nearrow$ \\ \hline
			$\omega_{4,h}$&2.852689 & 2.908394 & 2.919911 & $\nearrow$ \\ \hline
			$\omega_{5,h}$&3.008566 & 3.048110 & 3.051157 & $\nearrow$ \\ \hline
			\multicolumn{5}{|c|}{$\nu=0.499999$}\\ \hline
			$\omega_{1,h}$&0.685894 & 0.698333 & 0.701101 & $\nearrow$ \\ \hline
			$\omega_{2,h}$&1.823032 & 1.842722 & 1.847663 & $\nearrow$ \\ \hline
			$\omega_{3,h}$&1.841194 & 1.863939 & 1.865548 & $\nearrow$ \\ \hline
			$\omega_{4,h}$&2.852635 & 2.908339 & 2.919861 & $\nearrow$ \\ \hline
			$\omega_{5,h}$&3.008646 & 3.048190 & 3.051238 & $\nearrow$ \\ \hline	
		\end{tabular}
	\end{center}
\end{table}

\begin{table}[H]
	\begin{center}
		\caption{$WG$ method, k=2}\label{t5}
		\begin{tabular}{|c|c|c|c|c|}\hline
			$H$&1/8&1/16&1/32&Trend\\ \hline
			$h$&1/16&1/64&1/256&\\ \hline\hline
			\multicolumn{5}{|c|}{$\nu=0.49$}\\ \hline
			$\omega_{1,h}$&0.696634 & 0.698242 & 0.698959 & $\nearrow$ \\ \hline
			$\omega_{2,h}$&1.832686 & 1.835234 & 1.836336 & $\nearrow$ \\ \hline
			$\omega_{3,h}$&1.860501 & 1.860784 & 1.860811 & $\nearrow$ \\ \hline
			$\omega_{4,h}$&2.922405 & 2.925409 & 2.926599 & $\nearrow$ \\ \hline
			$\omega_{5,h}$&3.043038 & 3.043359 & 3.043388 & $\nearrow$ \\ \hline
			\multicolumn{5}{|c|}{$\nu=0.4999$}\\ \hline
			$\omega_{1,h}$&0.698486 & 0.700952 & 0.701445 & $\nearrow$ \\ \hline
			$\omega_{2,h}$&1.843553 & 1.847477 & 1.848255 & $\nearrow$ \\ \hline
			$\omega_{3,h}$&1.865210 & 1.865560 & 1.865565 & $\nearrow$ \\\hline
			$\omega_{4,h}$&2.915930 & 2.921036 & 2.922218 & $\nearrow$ \\ \hline
			$\omega_{5,h}$&3.050911 & 3.051287 & 3.051293 & $\nearrow$ \\ \hline
		\end{tabular}
	\end{center}
\end{table}

\begin{example}\label{e3}
	Consider the linear elastic eigenvalue problem \eqref{eig} on L-shaped  domain $\Omega=(0,2)^2/(1,2)^2$ with $\Gamma_N=\phi$. We set the Poisson ratio $\nu$=0.49, 0.4999, 0.499999, and solve it by Algorithm \ref{32}. The corresponding results are shown in table \ref{t6}.
\end{example}
\begin{table}[H]
	\begin{center}
		\caption{}\label{t6}
		\begin{tabular}{|c|c|c|c|c|c|c|}\hline
			$h$&1/16&1/32&1/64&1/128&1/256&Trend\\ \hline\hline
		\multicolumn{7}{|c|}{$\nu=0.49$}\\ \hline
			$\omega_{1,h}$&3.161484 & 3.227793 & 3.252536 & 3.262002 & 3.265820 & $\nearrow$ \\ \hline
			$\omega_{2,h}$&3.400566 & 3.477013 & 3.499606 & 3.506059 & 3.507882 & $\nearrow$ \\ \hline
			$\omega_{3,h}$&3.628877 & 3.693181 & 3.710797 & 3.715520 & 3.716802 & $\nearrow$ \\ \hline
			$\omega_{4,h}$&3.944009 & 4.015067 & 4.035357 & 4.040783 & 4.042192 & $\nearrow$ \\ \hline
			$\omega_{5,h}$&4.060406 & 4.169453 & 4.200536 & 4.209421 & 4.212065 & $\nearrow$ \\ \hline
		\multicolumn{7}{|c|}{$\nu=0.4999$}\\ \hline
			$\omega_{1,h}$&3.161870 & 3.230057 & 3.255689 & 3.265574 & 3.269595 & $\nearrow$ \\ \hline
			$\omega_{2,h}$&3.405130 & 3.481293 & 3.503792 & 3.510213 & 3.512022 & $\nearrow$ \\ \hline
			$\omega_{3,h}$&3.657156 & 3.716070 & 3.732610 & 3.737113 & 3.738361 & $\nearrow$ \\ \hline
			$\omega_{4,h}$&3.942279 & 4.013079 & 4.033424 & 4.038885 & 4.040309 & $\nearrow$ \\ \hline
			$\omega_{5,h}$&4.169090 & 4.258356 & 4.285648 & 4.293773 & 4.296312 & $\nearrow$ \\ \hline
		\multicolumn{7}{|c|}{$\nu=0.499999$}\\ \hline
			$\omega_{1,h}$&3.161864 & 3.230065 & 3.255704 & 3.265593 & 3.269616 & $\nearrow$ \\ \hline
			$\omega_{2,h}$&3.405161 & 3.481321 & 3.503820 & 3.510240 & 3.512049 & $\nearrow$ \\ \hline
			$\omega_{3,h}$&3.657243 & 3.716136 & 3.732671 & 3.737173 & 3.738421 & $\nearrow$ \\ \hline
			$\omega_{4,h}$&3.942252 & 4.013050 & 4.033395 & 4.038857 & 4.040281 & $\nearrow$ \\ \hline
			$\omega_{5,h}$&4.169254 & 4.258503 & 4.285791 & 4.293915 & 4.296454 & $\nearrow$ \\ \hline
		\end{tabular}
	\end{center}
\end{table}
	
	\bibliographystyle{siam}
\bibliography{library}
\end{document}